\documentclass[12pt,letterpaper]{article}
\usepackage[hmargin=1in,vmargin=1in]{geometry}

\usepackage{amsmath}
\usepackage{amsfonts}
\usepackage{amsthm}
\usepackage{amssymb}
\usepackage{array}
\usepackage{caption}
\usepackage{color}
\usepackage{enumitem}
\usepackage{epic,eepic}
\usepackage{graphics}
\usepackage{graphicx}
\usepackage{hyperref}
\usepackage{mathtools}
\usepackage{rotating}
\usepackage{verbatim}

\newtheorem{prop}{Proposition}
\newtheorem{cor}[prop]{Corollary}
\newtheorem{theorem}[prop]{Theorem}
\newtheorem{lemma}[prop]{Lemma}
\newtheorem{conj}[prop]{Conjecture}

\usepackage{tikz}
\usepackage{float}
\usetikzlibrary{shapes, arrows, positioning, snakes}

\begin{document}

\title{Conway's subprime Fibonacci sequences}
\author{Richard K.~Guy, Tanya Khovanova, Julian Salazar}
\maketitle

\centerline{\textit{In memory of Martin Gardner}}

\begin{abstract}
It's the age-old recurrence with a twist: sum the last two terms and \emph{if the result is composite, divide by its smallest prime divisor} to get the next term (e.g., 0, 1, 1, 2, 3, 5, 4, 3, 7, \ldots). These sequences exhibit pseudo-random behaviour and generally terminate in a handful of cycles, properties reminiscent of $3x + 1$ and related sequences. We examine the elementary properties of these `subprime' Fibonacci sequences.
\end{abstract}

\section{Introduction}
\label{sec:intro}

When John Conway last visited the first author, he passed the time on the plane by calculating what we now call \textbf{subprime Fibonacci sequences}. They are just the sort of thing Martin Gardner would have featured in his column. There is some risk of their becoming as notorious as the $3x+1$ (Collatz) problem \cite{La10}, with which they seem to have something in common, and of which Erd\H{o}s has said, ``Mathematics is not yet ripe for such problems.''

The $3x+1$ sequences take a positive integer and iteratively apply the following rule: if a number is odd, triple it and add one; if even, halve it:
\[
t_{n+1} = 
  \begin{cases}
   3t_n+1 & \text{if } t_n \text{ odd} \\
   \frac{t_n}{2} & \text{if } t_n \text{ even.}
  \end{cases}
\]
The sequences produced by this rule always appear to reach an infinite string of $4$, $2$, $1$, $4$, $2$, $1$, etc., and the problem is whether all sequences reach this \textbf{cycle}, i.e., whether for all $t_0$, there is some $n$ where $t_n = 1$. Here are some examples:
\begin{align*}
& 6, 3, 10, 5, 16, 8, \textbf{4, 2, 1}, 4, \ldots \\
& 17, 52, 26, 13, 40, 20, 10, 5, 16, 8, \textbf{4, 2, 1}, 4, \ldots \\
& 30, 15, 46, 23, 70, 35, 106, 53, 160, 80, 40, 20, 10, 5, 16, 8, \textbf{4, 2, 1}, 4, \ldots.
\end{align*}

Despite the simple rule, the paths of the sequences are rather unpredictable. Starting with 33 takes 26 steps and climbs to 100 before reaching 1, while 27 takes 111 steps and climbs to over 9000 before reaching 1. Such behavior has made this and other similar problems seem intractable \cite{Gu83}; we cannot even show that such sequences could not go to infinity. As Lagarias introduces the problem in his $3x+1$ compendium \cite{La10}, he states that it touches number theory, ergodic theory, stochastic processes, and more, while not lying squarely in any of their domains.

A more recreational example is given by Conway's RATS sequences, one of many base-dependent `reversal' sequences \cite{Gu89}. RATS stands for Reverse, Add, Then Sort: take a number with digits in increasing order, reverse it, add to the original number, and then sort the result's digits in increasing order. Here are some base-10 examples.
\begin{align*}
& 12334444, 55667777, 123334444, 556667777, 1233334444, 5566667777, 12333334444, \ldots \\
& 123, 444, 888, 1677, 3489, \textbf{12333, 44556, 111, 222, 444, 888, 1677, 3489}, 12333, \ldots
\end{align*}
The first sequence is known as \textbf{the creeper} (\href{http://oeis.org/A164338}{A164338} in OEIS\cite{OEIS}). It provably diverges in this regular pattern, and is reached by various starting terms such as $1$. Conway's conjecture is that all base-10 RATS sequences enter cycles (as in the second sequence) or enter the creeper and diverge.

One natural approach in tackling these types of problem involves restricting possible end behaviors of such sequences; their destinies, so to speak \cite{Kh09}. These two classes of sequence have rather different fates. For example, e Silva has verified that $3x+1$ sequences reach $1$ for starting numbers less than $5.76 \times 10^{18}$ \cite{La10}, and Simons and de Weger proved that if there were another $3x+1$ cycle it would have at least $69$ terms \cite{Si05}. By contrast, Cooper and Kennedy have shown the existence of base-10 RATS cycles for every length 2 and greater \cite{Co99}. Regardless, there are limits on potential analysis: Kurtz and Simon, building on earlier work by Conway, proved that a natural generalization of the $3x+1$ problem is undecidable \cite{Kur07}.

It is easy to discount these results as too problem-specific, and that such sequences could never lead to `useful' mathematics. Yet the appeal of such problems (the $3x+1$ problem was once called by S.\  Kakutani ``a conspiracy to slow down mathematical research in the U.S.''\cite[p.32]{La10}) has always lain in the contrast between how easy they are to play with and how hard it is to answer their questions. We hope the subprime Fibonacci sequences continue this tradition.

\section{Subprime Fibonacci sequences}

Start with the Fibonacci sequence 0, 1, 1, 2, 3, 5, \ldots, but before you write down a composite term, divide it by its least prime factor so that this next term is not 8, but rather $8/2 = 4$. After that the sum gives us $5+4=9$, but we write $9/3 = 3$, then $4+3=7$ which is okay since it is prime, then $3+7=10$ but we write $10/2=5$, and so on:

\begin{center}
\begin{tabular}{c@{\hspace{2mm}}c@{\hspace{2mm}}c@{\hspace{2mm}}c
@{\hspace{2mm}}c@{\hspace{2mm}}c@{\hspace{2mm}}c
@{\hspace{2mm}}c@{\hspace{2mm}}c@{\hspace{2mm}}c
@{\hspace{2mm}}c@{\hspace{2mm}}c@{\hspace{2mm}}c
@{\hspace{2mm}}c@{\hspace{2mm}}c@{\hspace{2mm}}c
@{\hspace{2mm}}c@{\hspace{2mm}}c}

0 & 1 & 1 & 2 &  3 & 5 & 4 & 3 & 7 & 5 & 6 & 11 & 17 & 14 & 31 & 15 & 23 & 19 \\
21 & 20 & 41 & 61 & 51 & 56 & 107 & 163 & 135 & 149 & 142 & 97 & 239 & 168 & 37 & 41 & 39 & 40 \\
79 & 17 & 48 & \textbf{13} & \textbf{61} & \textbf{37} & \textbf{49} & \textbf{43} & \textbf{46} & \textbf{89} & \textbf{45} & \textbf{67} & \textbf{56} & \textbf{41} & \textbf{97} & \textbf{69} & \textbf{83} & \textbf{76} \\
\textbf{53} & \textbf{43} & \textbf{48} & 13 & 61 & 37  & \ldots

\end{tabular}
\end{center}   
and we are in an 18-cycle. If we start with $1,1$ or $1,2$ it follows that we get the same result.  But we may start with any pair of numbers, and you may like to try starting with $2,1$, or $1,3$, or $3,9$, or $13,11$, etc.

One might suspect that every such sequence enters this 18-cycle, similar to the $3x+1$ problem's conjecture. After all, since our sequences are bounded or unbounded they must either enter a cycle or increase indefinitely. We do not believe the latter happens and provide a heuristic argument in Section \ref{sec:endconds}. But is the 18-cycle the only `non-trivial' cycle? Wait and see.

First, note that $a,a$, where $a \ne \pm1$ gives the sequence $a,a,a,a,\dotsc$. This is a \textbf{trivial cycle}. Sequences that end in trivial cycles are \textbf{trivial sequences}, e.g., 5, 15, 10, 5, 5, 5, \ldots, or $-143$, 39, $-52$, $-13$, $-13$, $-13$, \ldots. If two consecutive terms have the same sign then so do all subsequent terms. If they have opposite sign or include a zero, they bound further terms until two consecutive terms of the same sign appear, e.g., $-17$, 7, $-5$, 2, $-3$, $-1$, $-2$, \ldots, after which the sign remains constant.

Next, two terms of opposite parity are followed by an odd term, and two odd terms are followed by an even or an odd term depending on whether their sum is a multiple of 4. One can have arbitrarily long strings of even terms, but they must terminate since the power of 2 in consecutive terms must eventually decrease, e.g., 128, 160, 144, 152, 148, 150, 149, \ldots, and once we have an odd term (unless this sequence is trivial), subsequent even terms are isolated with each followed by at least two odd terms. Therefore, \emph{we are only concerned with sequences of positive terms, comprised of `runs' of odd terms separated by even terms}.

Finally, let the \textbf{shape} of a sequence be the string of its terms' parities ($O$ for odd, $E$ for even). The Fibonacci sequence has shape $EOOEOOEOOEOO\dotso$. Our first subprime Fibonacci sequence had shape $EOOEOOEOOOEOO\dotso$. The `extra' odd term here came from where the sum of the previous two odd terms only had one factor of 2. The example, starting at 13, 61 inclusive, gives the shape $OOOOOEOOOEOOOOEOOE$ that repeats with the 18-cycle.

\section{Nodes and other cycles}
\label{sec:nodes}

To help develop the terminology and flavor of these sequences, we plot their trajectories on a directed graph. This visual approach is often used in expositions of the $3x+1$ problem \cite[p.62]{La10}, with sequences as \textbf{paths} in an infinite digraph (Figure~\ref{fig:collatz}), determined by their starting points. The problem is whether this digraph is weakly connected (connected when viewed as an undirected graph).

\begin{figure}[H]
\centering

\begin{tikzpicture}[->,>=stealth',shorten >=1pt,auto,node distance=1.6cm,scale=0.8,transform shape,
  thick]

\node (c1) {4};
\node (c2) [above right of=c1] {2};
\node (c3) [below right of=c1] {1};
\node (1) [left of=c1] {8};
\node (2) [left of=1] {16};
\node (a) [left of=2] {32};
\node (b) [below left of=2] {5};
\node (a1) [left of=a] {64};
\node (aa) [above left of=a1] {21};
\node (ab) [left of=a1] {128};
\node (b1) [left of=b] {10};
\node (ba) [left of=b1] {20};
\node (bb) [below left of=b1] {3};
\node (aa1) [left of=aa] {42};
\node (ab1) [left of=ab] {256};
\node (ba1) [left of=ba] {40};
\node (bb1) [left of=bb] {6};
\node (aa2) [above left of=aa1] {...};
\node (aba) [above left of=ab1] {...};
\node (abb) [left of=ab1] {...};
\node (baa) [left of=ba1] {...};
\node (bab) [below left of=ba1] {...};
\node (bb2) [below left of=bb1] {...};

\path[every node/.style={font=\sffamily\small}]
(c1) edge [bend left] (c2)
(c2) edge [bend left] (c3)
(c3) edge [bend left] (c1)
(1) edge (c1)
(2) edge (1)
(a) edge (2)
(b) edge (2)
(a1) edge (a)
(aa) edge (a1)
(ab) edge (a1)
(b1) edge (b)
(ba) edge (b1)
(bb) edge (b1)
(aa1) edge (aa)
(ab1) edge (ab)
(ba1) edge (ba)
(bb1) edge (bb)
(aa2) edge (aa1)
(aba) edge (ab1)
(abb) edge (ab1)
(baa) edge (ba1)
(bab) edge (ba1)
(bb2) edge (bb1);

\end{tikzpicture}

\caption{Digraph generated by the $3x+1$ sequences}
\label{fig:collatz}

\end{figure}
 
   However, our sequences cannot immediately be represented in this fashion because of the second-order nature of our recurrence. We must carefully define vertices for our sequences, and so we introduce two important terms:
\begin{itemize}
\item The \textbf{nodes} of a sequence are ordered pairs of positive, odd coprimes which either begin the sequence or immediately follow the even terms of a sequence.
\item \textbf{Runs} are the strings beginning with a node and consisting of odd terms together with a single terminating even term.
\end{itemize}
In Section \ref{sec:endconds} we will see that every non-trivial sequence becomes composed of runs after some point. Here is our initial sequence with nodes parenthesized:

\begin{center}
\begin{tabular}{c@{\hspace{2mm}}c@{\hspace{2mm}}c@{\hspace{2mm}}c
@{\hspace{2mm}}c@{\hspace{2mm}}c@{\hspace{2mm}}c
@{\hspace{2mm}}c@{\hspace{2mm}}c@{\hspace{2mm}}c
@{\hspace{2mm}}c@{\hspace{2mm}}c@{\hspace{2mm}}c
@{\hspace{2mm}}c@{\hspace{2mm}}c@{\hspace{2mm}}c
@{\hspace{2mm}}c@{\hspace{2mm}}c}

0 & (1, & 1) & 2 &  (3, & 5) & 4 & (3, & 7) & 5 & 6 & (11, & 17) & 14 & (31, & 15) & 23 & 19 \\
21 & 20 & (41, & 61) & 51 & 56 & (107, & 163) & 135 & 149 & 142 & (97, & 239) & 168 & (37, & 41) & 39 & 40 \\
(79, & 17) & 48 & \textbf{(13,} & \textbf{61)} & \textbf{37} & \textbf{49} & \textbf{43} & \textbf{46} & \textbf{(89,} & \textbf{45)} & \textbf{67} & \textbf{56} & \textbf{(41,} & \textbf{97)} & \textbf{69} & \textbf{83} & \textbf{76} \\
\textbf{(53,} & \textbf{43)} & \textbf{48} & (13, & 61) & 37  & \ldots

\end{tabular}
\end{center}

Again, we can treat each substring of $O \dotso OE$ as a unit, starting when the first two terms of such a substring are coprime and not preceded by an odd term. The corresponding terms comprise a run, and the first two terms of the run comprise a node. Let us now construct our first sequence path (Figure~\ref{fig:digraph1}). For notational convenience we weight the digraph by assigning to each arc the length of the run generated by the node at the arc's tail.

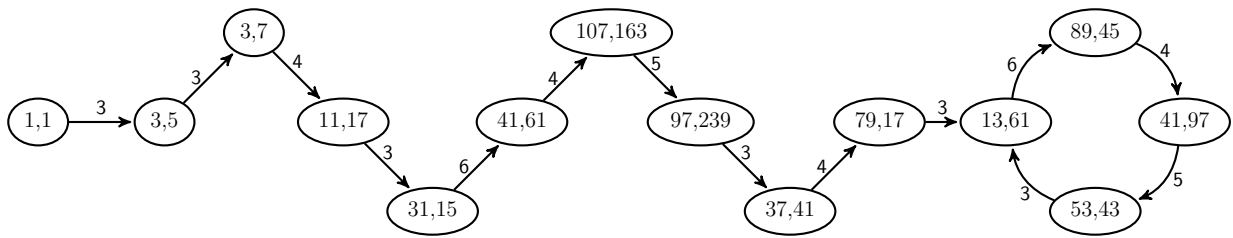
\begin{figure}[h]
\centering

\begin{tikzpicture}[->,>=stealth',shorten >=1pt,auto,node distance=2.4cm,scale=0.7,transform shape,
  thick,main node/.style={ellipse,draw}]

\node[main node] (0) {1,1};
\node[main node] (1) [right of=0] {3,5};
\node[main node] (2) [above right of=1] {3,7};
\node[main node] (3) [below right of=2] {11,17};
\node[main node] (4) [below right of=3] {31,15};
\node[main node] (5) [above right of=4] {41,61};
\node[main node] (6) [above right of=5] {107,163};
\node[main node] (7) [below right of=6] {97,239};
\node[main node] (8) [below right of=7] {37,41};
\node[main node] (9) [above right of=8] {79,17};
\node[main node] (c1) [right of=9] {13,61};
\node[main node] (c2) [above right of=c1] {89,45};
\node[main node] (c3) [below right of=c2] {41,97};
\node[main node] (c4) [below left of=c3] {53,43};

\path[every node/.style={font=\sffamily\small}]
(0) edge node [above] {3} (1)
(1) edge node [left] {3} (2)
(2) edge node [above] {4} (3)
(3) edge node [above] {3} (4)
(4) edge node [left] {6} (5)
(5) edge node [left] {4} (6)
(6) edge node [above] {5} (7)
(7) edge node [above] {3} (8)
(8) edge node [left] {4} (9)
(9) edge node [above] {3} (c1)
(c1) edge [bend left] node [left] {6} (c2)
(c2) edge [bend left] node [above] {4} (c3)
(c3) edge [bend left] node [right] {5} (c4)
(c4) edge [bend left] node [below] {3} (c1);

\end{tikzpicture}

\caption{Path generated by the 0,1 sequence}
\label{fig:digraph1}

\end{figure}

We could then imagine the infinite digraph generated by all non-trivial subprime Fibonacci sequences, as we have done for the $3x+1$ sequences in Figure~\ref{fig:collatz}. If the 18-cycle were the only non-trivial cycle, the subprime Fibonacci digraph would look like Figure \ref{fig:tributaries}.

\begin{figure}[H]
\centering

\begin{tikzpicture}[->,>=stealth',shorten >=1pt,auto,node distance=2.4cm,scale=0.7,transform shape,
  thick,main node/.style={ellipse,draw}]

\node[main node] (c1) {13,61};
\node[main node] (c2) [above right of=c1] {89,45};
\node[main node] (c3) [below right of=c2] {41,97};
\node[main node] (c4) [below left of=c3] {53,43};

\node[main node] (a1) [left of=c1] {79,17};
\node[main node] (a2) [below left of=a1] {37,41};
\node (a3) [left of=a2] {...};

\node[main node] (b2) [left of=a1] {61,29};
\node[main node] (b3) [left of=b2] {53,17};

\node[main node] (d3) [above left of=b2] {373,37};

\node[main node] (e1) [above left of=c2] {49,43};
\node[main node] (e2) [left of=e1] {41,93};
\node (e3) [left = 0.6 cm of e2] {...};

\node[main node] (f1) [above right of=c2] {109,13};
\node[main node] (f2) [right of=f1] {27,71};
\node[main node] (f3) [right of=f2] {79,23};
\node (f4) [below right of=f3] {...};

\node[main node] (g2) [below right of=f1] {71,49};

\node[main node] (h1) [below right of=c3] {45,67};

\node[main node] (i1) [right = 0.7 cm of c3] {971,305};
\node (i2) [right = 0.7 cm of i1] {...};

\node[main node] (j1) [below left of=c4] {69,83};

\node[main node] (k1) [below right of=c4] {2027,1031};
\node (k2) [right = 0.8 cm of k1] {...};

\path[every node/.style={font=\sffamily\small}]
(a3) edge node [above] {3} (a2)
(a2) edge node [left] {4} (a1)
(a1) edge node [above] {3} (c1)
(b3) edge node [above] {4} (b2)
(b2) edge node [above] {7} (a1)
(d3) edge node [above] {6} (b2)
(e3) edge node [above] {4} (e2)
(e2) edge node [above] {4} (e1)
(e1) edge node [above] {3} (c2)
(f4) edge node [right] {3} (f3)
(f3) edge node [above] {5} (f2)
(f2) edge node [above] {4} (f1)
(f1) edge node [above] {7} (c2)
(g2) edge node [right] {3} (f1)
(h1) edge node [right] {3} (c3)
(i2) edge node [above] {4} (i1)
(i1) edge node [above] {3} (c3)
(j1) edge node [left] {3} (c4)
(k2) edge node [above] {4} (k1)
(k1) edge node [right] {4} (c4)
(c1) edge [bend left] node [left] {6} (c2)
(c2) edge [bend left] node [above] {4} (c3)
(c3) edge [bend left] node [right] {5} (c4)
(c4) edge [bend left] node [below] {3} (c1);

\end{tikzpicture}

\caption{Some paths leading to the 18-cycle}
\label{fig:tributaries}

\end{figure}
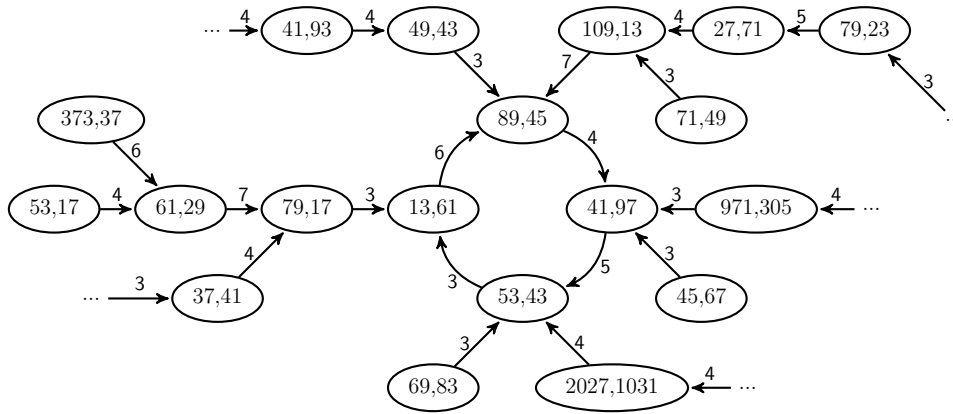

One reason this digraph is a nice representation is that it shows how many nodes are \textbf{direct predecessors} to a single node. If a node is a predecessor (not necessarily direct) to a node or cycle, we say it is \textbf{tributary} to the node or cycle. How could we grow this graph? One way is to go outwards from known nodes. This would require a way of enumerating a node's direct predecessors, which can be done with some work. For example, with the node $(89,45)$ of the 18-cycle:

\begin{enumerate}
\item A preceding even term $t$ must satisfy $t + 89 = 45q$, where $q$ is 1 or 3 ($q = 2$ makes $t$ odd, and $q$ cannot exceed a prime factor of 45), which gives $t = -44$ or 46, so the node must always be preceded by 46.
\item Let the positive odd term before $t$, if it exists, be $s$. Then $s + t = 89p$, where $p$ is 1 or an odd prime $\le 89$. For $t=46$, possible values of $s$ are $43$, 221, 399, etc.
\item The term before $s$ must also be odd. If this term is $r$, it must satisfy $r + s = 2t$ since $r+s$ is even. For example, $s=43$ gives $r=49$, and none of the other possibilities for $s$ would work, since they would make $r\le0$.
\item Since we are only looking for possible direct predecessors (positive, odd coprimes), we can assume that each prior step involved division by two. Working backwards gives
\[
\ldots,\:-83,\:109,\:13,\:61,\:37,\:49,\:43,\:46.
\]
\item Thus, our direct predecessors are exactly
\[
(109,13),\:(13,61),\:(61,37),\:(37,49),\:\text{and}\:(49,43),
\]
only two of which are depicted in Figure~\ref{fig:tributaries}. 
\end{enumerate}

Contrast this with Figure~\ref{fig:collatz}, where there are at most two direct predecessors as a result of the sequence definition. This procedure for constructively generating nodes is quite finicky, however, and discourages a graph-theoretic approach to analysis. Not to say that it is impossible; there exist reductions and results on the $3x+1$ graph \cite{And02} \cite{Urv00}, and we encourage the reader to explore the possibility of deriving properties for the subprime digraph from this perspective.

Do sequences all enter the 18-cycle we have already seen, i.e., is the subprime digraph weakly connected?  Let us start at the node $(151,227)$:

\begin{center}
\begin{tabular}{c@{\hspace{3pt}}c@{\hspace{3pt}}c@{\hspace{3pt}}c
@{\hspace{3pt}}c@{\hspace{3pt}}c@{\hspace{3pt}}c
@{\hspace{3pt}}c@{\hspace{3pt}}c@{\hspace{3pt}}c
@{\hspace{3pt}}c@{\hspace{3pt}}c@{\hspace{3pt}}c
@{\hspace{3pt}}c@{\hspace{3pt}}c@{\hspace{3pt}}c
@{\hspace{3pt}}c@{\hspace{3pt}}c@{\hspace{3pt}}c}

(151, & 227) & 189 & 208 & (397, & 121) & 259 & 190 & (449, & 213) & 331 & 272 & (201, & 43) & 122 & (55, & 59) & 57 & 58\\
\textbf{(23,} & \textbf{27)} & \textbf{25} & \textbf{26} & \textbf{(17,} & \textbf{43)} & \textbf{30} & \textbf{(73,} & \textbf{103)} & \textbf{88} & \textbf{(191,} & \textbf{93)} & \textbf{142} & \textbf{(47,} & \textbf{63)} & \textbf{55} & \textbf{59} & \textbf{57} & \textbf{58}\\
(23, & 27) & \ldots

\end{tabular}
\end{center} 
and we are in a 19-cycle whose first repeated node is $(23,27)$. Note that though 55, 59 are the first two repeated terms, they only act as a node the first time through; thus $(47,63)$ being a node with the terms 55, 59 in its run does not preclude $(55,59)$ from being a node in another context. Both nodes are tributary to the node $(23,27)$. Furthermore,

\begin{itemize}
\item If you start with $5,13$ you will enter a 136-cycle through node $(47,23)$ (though simpler starting terms like $1,4$ suffice).
\item If you start with $5, 23$ you will enter a 56-cycle through node $(119,109)$ with 5693 as its largest term.
\item The node $(37,199)$ generates an 11-cycle.
\item The node $(127,509)$ generates a 10-cycle.
\end{itemize}

Figure~\ref{fig:cycles} displays the nodes in these non-trivial cycles. We checked sequences that start with two numbers 1,000,000 or below and found no non-trivial cycles other than these six, a bound easily extendable by our more computationally-minded readers.

\begin{figure}[H]

\centering

\begin{tikzpicture}[->,>=stealth,scale=0.75,transform shape,thick,main node/.style={ellipse,draw}]

\node[main node] (a1) at (19,1.5) {47,23};
\node[main node] (a2) at (16.5,0) {61,31};
\node[main node] (a3) at (13.5,0) {11,19};
\node[main node] (a4) at (10.75,0) {11,9};
\node[main node] (a5) at (8,0) {19,29};
\node[main node] (a6) at (5.25,0) {53,11};
\node[main node] (a7) at (2.5,0) {43,25};
\node[main node] (a8) at (0,1.5) {59,31};
\node[main node] (a9) at (0,3) {83,11};
\node[main node] (a10) at (0,4.5) {67,35};
\node[main node] (a11) at (0,6) {13,59};
\node[main node] (a12) at (0,7.5) {19,11};
\node[main node] (a13) at (0,9) {9,23};
\node[main node] (a14) at (0,10.5) {13,29};
\node[main node] (a15) at (0,12) {47,71};
\node[main node] (a16) at (0,13.5) {127,63};
\node[main node] (a17) at (0,15) {13,97};
\node[main node] (a18) at (2.5,15) {131,69};
\node[main node] (a19) at (5.25,15) {13,113};
\node[main node] (a20) at (8,15) {151,239};
\node[main node] (a21) at (11,15) {141,347};
\node[main node] (a22) at (13.85,15) {197,147};
\node[main node] (a23) at (16.5,15) {29,67};
\node[main node] (a24) at (19,15) {23,71};
\node[main node] (a25) at (19,13.5) {109,55};
\node[main node] (a26) at (19,12) {137,73};
\node[main node] (a27) at (19,10.5) {63,157};
\node[main node] (a28) at (19,9) {89,199};
\node[main node] (a29) at (19,7.5) {49,193};
\node[main node] (a30) at (19,6) {41,63};
\node[main node] (a31) at (19,4.5) {23,25};
\node[main node] (a32) at (19,3) {7,31};

\node[main node] (b1) at (6,5.1) {23,27};
\node[main node] (b2) at (7.5,6.5)  {17,43};
\node[main node] (b3) at (10.25,6.5) {73,103};
\node[main node] (b4) at (12,5.1) {191,93};
\node[main node] (b5) at (9,4) {47,63};

\node[main node] (c1) at (3,3.75) {223,337};
\node[main node] (c2) at (3,5.25) {617,299};
\node[main node] (c3) at (3,6.75) {757, 405};
\node[main node] (c4) at (4,8.25) {347, 291};
\node[main node] (c5) at (7.5,8.25) {617, 929};
\node[main node] (c6) at (11,8.25)  {1663, 825};
\node[main node] (c7) at (14.5,8.25)  {2069, 3313};
\node[main node] (c8) at (15.5,6.75) {5693, 1739};
\node[main node] (c9) at (15.5,5.25) {1091, 437};
\node[main node] (c10) at (15.5,3.75) {1201, 655};
\node[main node] (c11) at (14.5,2.25) {1583, 837};
\node[main node] (c12) at (11,2.25) {89, 433};
\node[main node] (c13) at (7.5,2.25) {217, 521};
\node[main node] (c14) at (4, 2.25) {119, 109};

\node[main node] (d1) at (3,10.5) {37,199};
\node[main node] (d2) at (3,13.5)  {317,145};
\node[main node] (d3) at (4.75,12) {419,607};

\node[main node] (e1) at (15.5,10.5) {127,509};
\node[main node] (e2) at (13.75,12) {827,229};
\node[main node] (e3) at (15.5,13.5)  {757,257};

\node[main node] (f1) at (8,13.5) {13,61};
\node[main node] (f2) at (10.5, 13.5) {89,45};
\node[main node] (f3) at (10.5,10.5) {41,97};
\node[main node] (f4) at (8,10.5) {53,43};

\node[font=\fontsize{16}{16}\bfseries] at (2.5, 1) {136-cycle};
\node[font=\fontsize{16}{16}\bfseries] at (5.5, 3.25) {56-cycle};
\node[font=\fontsize{16}{16}\bfseries] at (9, 5.25) {19-cycle};
\node[font=\fontsize{16}{16}\bfseries] at (3.5, 9.45) {11-cycle};
\node[font=\fontsize{16}{16}\bfseries] at (9.25, 9.45) {18-cycle};
\node[font=\fontsize{16}{16}\bfseries] at (15, 9.45) {10-cycle};

\path[every node/.style={font=\sffamily\small}]

(a1) edge [bend left] node [right] {5} (a2)
(a2) edge node [below] {3} (a3)
(a3) edge node [below] {5} (a4)
(a4) edge node [below] {3} (a5)
(a5) edge node [below] {3} (a6)
(a6) edge node [below] {3} (a7)
(a7) edge [bend left] node [below] {3} (a8)
(a8) edge node [left]{4} (a9)
(a9) edge node [left] {5} (a10)
(a10) edge node [left] {7} (a11)
(a11) edge node [left] {3} (a12)
(a12) edge node [left] {5} (a13)
(a13) edge node [left] {3} (a14)
(a14) edge node [left] {6} (a15)
(a15) edge node [left] {5} (a16)
(a16) edge node [left] {8} (a17)
(a17) edge node [above] {4} (a18)
(a18) edge node [above] {3} (a19)
(a19) edge node [above] {4} (a20)
(a20) edge node [above] {5} (a21)
(a21) edge node [above] {3} (a22)
(a22) edge node [above] {3} (a23)
(a23) edge node [above] {3} (a24)
(a24) edge node [right] {6} (a25)
(a25) edge node [right] {3} (a26)
(a26) edge node [right] {8} (a27)
(a27) edge node [right] {3} (a28)
(a28) edge node [right] {3} (a29)
(a29) edge node [right] {6} (a30)
(a30) edge node [right] {3} (a31)
(a31) edge node [right] {3} (a32)
(a32) edge node [right] {5} (a1)

(b1) edge [bend left] node [left] {4} (b2)
(b2) edge node [above] {3} (b3)
(b3) edge [bend left] node [right] {3} (b4)
(b4) edge [bend left] node [below] {3} (b5)
(b5) edge [bend left] node [below] {6} (b1)

(c1) edge node [left] {3} (c2)
(c2) edge node [left] {3} (c3)
(c3) edge [bend left] node [left] {7} (c4)
(c4) edge node [above] {5} (c5)
(c5) edge node [above] {5} (c6)
(c6) edge node [above] {3} (c7)
(c7) edge [bend left] node [right] {4} (c8)
(c8) edge node [right] {3} (c9)
(c9) edge node [right] {3} (c10)
(c10) edge [bend left] node [right] {3} (c11)
(c11) edge node [above] {3} (c12)
(c12) edge node [above] {5} (c13)
(c13) edge node [above] {6} (c14)
(c14) edge [bend left] node [left] {3} (c1)

(d1) edge node [left] {3} (d2)
(d2) edge [bend left] node [right] {4} (d3)
(d3) edge [bend left] node [right] {4} (d1)

(e1) edge [bend left] node [left] {3} (e2)
(e2) edge [bend left] node [left] {3} (e3)
(e3) edge node [right] {4} (e1)

(f1) edge node [above] {6} (f2)
(f2) edge node [right] {4} (f3)
(f3) edge node [below] {5} (f4)
(f4) edge node [left] {3} (f1)
;

\end{tikzpicture}

\caption{Digraphs of the six known non-trivial cycles}
\label{fig:cycles}

\end{figure}
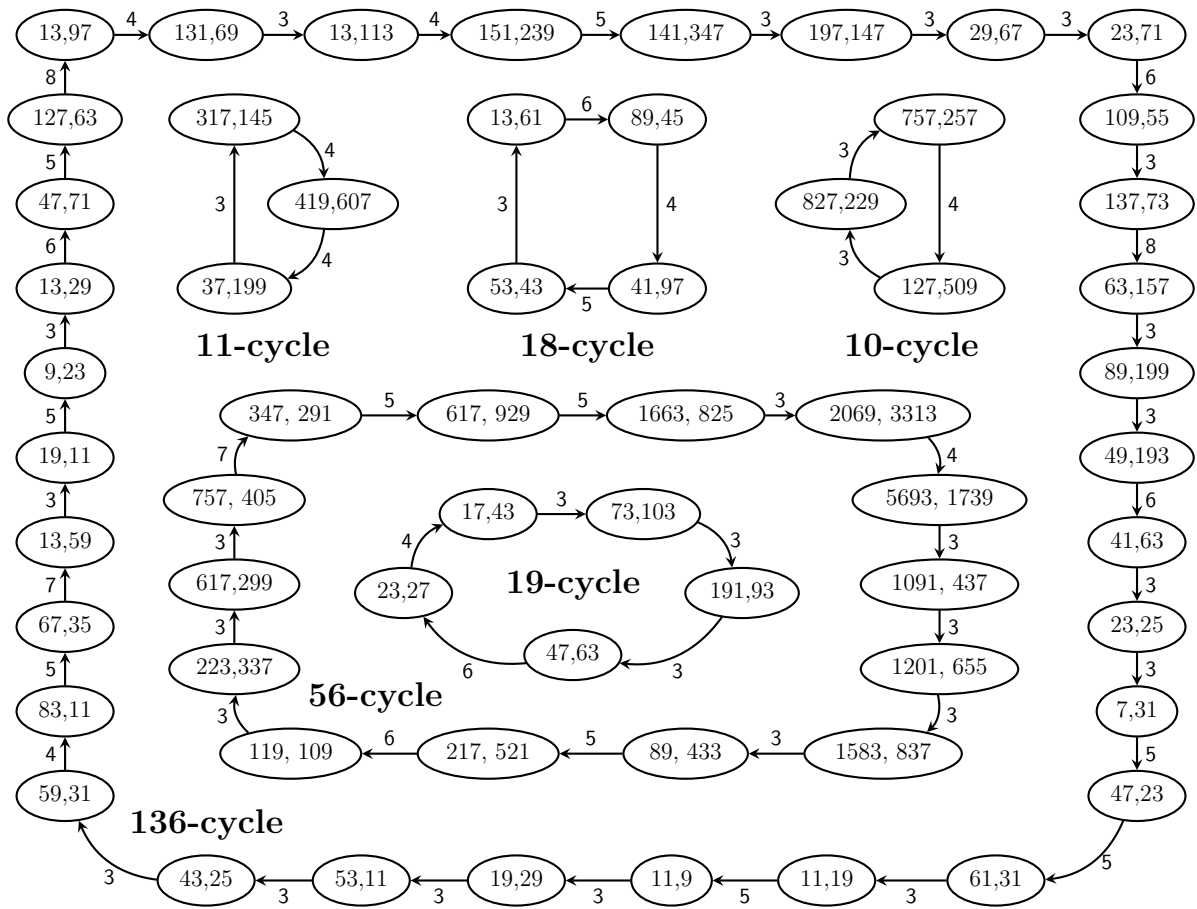

In Table~\ref{tab:cycledist} the headings indicate the range for the first two terms of the sequence and the entries are the number of occurrences for each cycle length. The proportion of pairs which generate each non-trivial cycle stabilizes as the range for starting terms increases. Additionally, non-trivial cycles appear to be distributed among the starting pairs rather arbitrarily. However, trivial cycles decrease in proportion since a cycle $a,a\dotsc$ requires all earlier terms to be multiples of $a$. Applying the `direct predecessor' method shows why this is, and how this makes relatively few starting conditions lead to a given trivial cycle.

\begin{table}[H]

\centering
\begin{tabular}{ |>{\bfseries} c |  l | l | l | l | l |}
\hline
\textbf{Cycle length} & $a,b \le10$ & $a,b \le 10^2$ & $a,b \le10^3$ & $a,b \le10^4$ & $a,b \le 10^5$ \\ \hline
1 & 14 & 348 & 10022 & 320531 & 11588563 \\ \hline
10 & 0 & 0 & 33 & 6310 & 668764 \\ \hline
11 & 0  & 0 & 390 & 34520 & 3479974 \\ \hline
18 & 63 & 4837 & 467014 & 46985673 & 4709133000 \\ \hline
19 & 0 & 249 & 30490 & 3090886 & 307710709 \\ \hline
56 & 0 & 188 & 21990 & 2238493 & 224936180 \\ \hline
136 & 23 & 4378 & 470061 & 47323587 & 4742482810 \\ \hline
\end{tabular}

\caption{Distribution of final cycle lengths generated by starting node $(a,b)$}
\label{tab:cycledist}

\end{table}

Generally, non-trivial sequences seem to exhibit pseudo-random behavior in their terms and their digraphs, regarding the length of their paths, the nodes they pass, and their associated cycles. We believe this is partly due to the construction, which relies on prime factorizations (the relationship between these factorizations and addition is not well understood). A similar difficulty is seen in the earlier RATS sequences, where the relationship between base-dependent reversal/sort and addition is essential to analysis.

However, another source of apparent randomness seems to be the iteration's conditionality itself, as with the $3x+1$ sequences. For example, if one considers a variant of subprime Fibonacci where only division by 5 occurs (when the sum is divisible by 5), similar observations as the above arise. We begin to feel the apparent intractability mentioned earlier of proving results on `destinies' of sequences like these.

\section{End conditions}
\label{sec:endconds}

A sequence must either end in a trivial cycle, a non-trivial cycle, or increase indefinitely. These \textbf{end conditions} are of interest; however, it seems more likely here than in the $3x+1$ problem that sequences do not increase indefinitely. Here is an informal argument that supports such a conjecture, which relies on the following observation:

\begin{prop}\label{runbounds}
The terms of the run defined by $(a,b)$ are bounded, above and below, in the interval $[a,b]$, and terms after the node are bounded in the interval $[\frac{M}{4},M]$ where $M = \max(a,b)$. In general, two consecutive run terms bound the rest of the run.
\end{prop}

This is because after the first two odds $a,b$ (the node terms) of a run, every successive term up to the terminating even is due to a division by 2; within the run we are averaging two consecutive terms at a time. Now, given the maximum of the current run $M$, what can we say about the maximum of the next run? Let $s,t \le M$ be the last two terms of the current run, and let $c,d$ be the first two terms of the next run. By Proposition~\ref{runbounds}, $t \ge \frac{M}{4}$. Remember that $t$ is even and $s$ is odd:

\begin{enumerate}
\item $c = s+t$ is prime, so $c < 2M$ (strict since $s+t$ is odd).
\begin{enumerate}
\item $t+c$ is prime, so $d = t+c < 3M$. The next run is bounded by $\max(c,d) < 3M$.
\item $t+c$ is composite, so $d \le \frac{t+c}{3} < M$. The next run is bounded by $\max(c,d) < 2M$.
\end{enumerate}
\item $s+t$ is composite, so $c \le \frac{s+t}{3} < \frac{2M}{3}$ (strict since $s+t$ is odd).
\begin{enumerate}
\item $t+c$ is prime, so $d = t+c < \frac{5M}{3}$. The next run is bounded by $\max(c,d) < \frac{5M}{3}$.
\item $t+c$ is composite, so $d \le \frac{t+c}{3} < \frac{5M}{9}$. The next run is bounded by $\max(c,d) < \frac{2M}{3}$.
\end{enumerate}
\end{enumerate}

What is the probability that $s+t$ or $t+c$ is prime? Assuming that our values are random and independent, the probability that each individually is prime is at most $\frac{1}{\ln{t}} \le \frac{1}{\ln{M/4}}$ by the prime number theorem. Then the next run is bounded by $\frac{2}{3}M$ with probability at least $\left(1-\frac{1}{\ln{M/4}}\right)^2$, which approaches 1 as $M$ increases.

Considering that starting terms $5, 23$ produce a term as large as 5693, the difficulty of proving the non-existence of divergent sequences seems comparable in difficulty to the same problem for the $3x+1$ sequences, and so we do not dwell on it further. However, there are other types of result we can prove about sequences' later conditions:

\begin{prop}\label{infprime}
A non-trivial sequence contains infinitely many primes (not necessarily distinct), each greater than its two preceding terms.
\end{prop}

\begin{proof}
If, after some point, the sequence contains no two consecutive terms that sum to a prime, then at each step a division happens and so the maximum value of any two consecutive terms decreases over time. Since this value cannot decrease forever, we get a contradiction. This proposition is stronger than only asserting infinitely many primes, as a prime could be generated after the division by a prime factor.
\end{proof}

\begin{prop}\label{coprime}
After some point consecutive terms of a non-trivial sequence are always coprime. 
\end{prop}

\begin{proof}
The greatest common divisor of two consecutive terms of the sequence cannot increase, since $\gcd(a,b) = \gcd(b,a+b)$ and so $\gcd\left(b,\frac{a+b}{p}\right) \leq \gcd(a,b)$. A non-trivial sequence must include a prime larger than its two preceding terms. Thereafter
the GCD of any two consecutive terms is 1.
\end{proof} 

\begin{cor}\label{coprimestart}
If the two starting terms of a sequence are coprime, the sequence is non-trivial.
\end{cor}

These results help justify our earlier definitions of node and run, as they ensure that only non-trivial sequences produce digraphs and that digraphs are unique representations (since a run uniquely leads into another run with no intervening terms, as evens cannot appear consecutively once nodes come into play).

Proposition~\ref{coprime} and Corollary~\ref{coprimestart} also simplify the search for cycles via starting conditions. Since all starting terms $a,b$ with $\gcd(a,b) >1$ produce trivial sequences or reach $\gcd(a,b) = 1$ for consecutive terms, it suffices to study starting conditions with $\gcd(a,b) =1$ to enumerate all non-trivial end conditions. Caching certain intermediate nodes and using a lookup table of primes provides a more efficient search method for new cycles than testing all positive integer ordered pairs.

\section{The general system}

We devote the rest of our paper to the cycles that non-trivial sequences generate. By the definition of a run, a non-trivial cycle must consist of a concatenation of runs. It follows from Proposition~\ref{coprime} that any two consecutive terms in a non-trivial cycle are coprime.

When we build a subprime Fibonacci sequence we add two numbers first then divide by a prime number or 1. Let us correspond to each term of a sequence or a cycle the smallest prime divisor (or 1) by which the sum of the two prior terms was divided. These divisors are the sequence's or cycle's \textbf{signature}. For example, the 10-cycle 127, 509, 318, 827, 229, 528, 757, 257, 507, 382 has signature 7, 1, 2, 1, 5, 2, 1, 5, 2, 2; the initial 7 is the divisor to get 127, after adding the preceding cycle terms 507, 382.

Signature terms can be relatively large; the 11-cycle has signature 29, 3, 2, 1, 3, 2, 2, 1, 1, 2, 2 (since one of the intermediate sums is $29\times37=1073$). Runs consist of consecutive averages, so given a run within a cycle, only its node (first two terms) has signature values not equal to 2. See this in action by noting the shape of the 10-cycle ($OOEOOEOOOE$) and comparing with the signature given earlier. Using Proposition~\ref{infprime}, this result on signatures follows:

\begin{cor}\label{largestterm}
The largest term of a cycle must be prime with a signature value of 1.
\end{cor}

\begin{proof}
Let $a$, $b$, $c$ be consecutive terms of a cycle. Then $c > \max(a,b)$ if and only if $c = a+b$, i.e., the signature value of $c$ is 1 and $c$ is prime.
\end{proof}

With the terms and the signature of a cycle, we can establish a homogeneous linear system. Let $t_1,\dotsc,t_m$ be the terms of the cycle and $s_1,\dotsc,s_m$ be the corresponding signature. Then \ldots, $t_{m-1} + t_{m} = s_1 t_1$, $t_{m} + t_1 = s_2 t_2$, $t_1 + t_2 = s_3 t_3$, \ldots, $t_{i-2} + t_{i-1} = s_i t_i$, \ldots. In matrix form,

\begin{equation}\label{eq:system}
\begin{bmatrix}
s_1 & 0 & 0 & \cdots & 0 & -1 & -1 \\
-1 & s_2 & 0 & \cdots & 0 & 0 & -1 \\
-1 & -1 & s_3 & \cdots & 0 & 0 & 0 \\
\vdots & \vdots & \vdots & \ddots & \vdots & \vdots & \vdots \\
0 & 0 & 0 & \cdots & s_{m-2} & 0 & 0 \\
0 & 0 & 0 & \cdots & -1 & s_{m-1} & 0 \\
0 & 0 & 0 & \cdots & -1 & -1 & s_{m} \\
\end{bmatrix}
\begin{bmatrix}
t_1 \\
t_2 \\
t_3 \\
\vdots \\
t_{m-2} \\
t_{m-1} \\
t_{m} \\
\end{bmatrix}
= \textbf{0}.
\end{equation}

\vspace{1em}

We can now relate signatures to cycles and begin restricting potential cycles:

\begin{theorem}\label{signaturetocycle}
No two cycles have the same signature.
\end{theorem}

\begin{proof}
This is equivalent to showing that a potential signature $s_1, \dotsc, s_m$ defines at most one cycle. Given a potential signature, consider solutions for $t_1,\dotsc, t_m$ over the reals. We have a system of $m$ linear homogeneous equations in $m$ variables. In this particular set of equations, all of the variables are expressible through exactly two consecutive ones, so the space of real solutions is at most 2-dimensional. Given consecutive terms $t_i,t_{i+1}$ and positive signature values, the equations must reduce to $A t_i+ B t_{i+1} = t_i$ and $C t_i + D t_{i+1} = t_{i+1}$ for some positive $A,B,C,D$. Thus $t_i$ is expressible through $t_{i+1}$ and the solution space is at most 1-dimensional.

If the solution is 1-dimensional, let one of the terms equal 1. The terms are in constant rational proportion to each other, so we can scale all the terms until the smallest set of integer solutions is produced. The largest term may be prime; this solution is potentially a cycle. Further scaling cannot produce another cycle since the largest term would not be prime (Corollary~\ref{largestterm}).
\end{proof}

\begin{theorem}\label{onerun}
There are no non-trivial cycles of one run (i.e., one even term).
\end{theorem}

\begin{proof}
Let $(t_1,t_2)$ be the node of the run. Sum all the row equations of Eq.\ \ref{eq:system} to get
\begin{equation*}
2(t_1 + \dotsb + t_m) = s_1 t_1 + \dotsb + s_m t_m.
\end{equation*}
By definition, $s_3,\dotsc,s_m = 2$, so the equation becomes $2(t_1+t_2) = s_1 t_1 + s_2 t_2$. Suppose first that $t_1$ is the largest prime with $s_1$ = 1. Then $t_1 = (s_2 - 2)t_2$. Since $t_2 > 1$ (dividing a composite number by its smallest prime factor will never produce 1) and $t_1$ is prime, $s_2=3$ and $t_1=t_2$, which is a contradiction since this is a non-trivial cycle. The argument is the same if $t_2$ is the largest prime.
\end{proof}

Since each run has at least 3 terms:

\begin{cor}
There are no non-trivial cycles of length below 6. If a cycle of length 6 exists, its shape must be $OOEOOE$.
\end{cor}

The trick of Theorem \ref{onerun} does not generalize to helping find cycles of more than one run. In this regard, we look to Theorem \ref{signaturetocycle} because it shows that results on signatures are necessarily results on cycles, which makes it desirable to relate signature terms within a cycle in a meaningful way. A signature is only useful if it produces a 1-dimensional solution space, requiring a determinant of 0.

One such relation could involve finding the general expression for the determinant of an $m$-cycle in terms of $s_1, \dotsc, s_m$, which we leave as an exercise to the reader. Our issue with this approach is that it ignores the run-based structure of cycles, and so we present a reduction where the only signature terms of interest are those corresponding to the node of each run (divisors not equal to 2).

\section{The run-centric system}

Constructing such a relation is powerful as it cements the correspondence between cycles and nodes, providing a more natural categorization of cycles. In Section \ref{sec:givenlen}, it lets us demonstrate an algorithm that disqualifies entire classes of cycle. This, combined with related signature restrictions in Section \ref{sec:sigres}, contributes to bounding future cycles by their lengths and shapes, as opposed to bounding the size of their terms. This is analogous to the two types of bound for the $3x+1$ problem: Simons and de Weger's lower bound on cycle lengths versus e Silva's lower bound on cycle term size, which restrict cycle classes and magnitudes respectively.

Before we begin, let us define the Jacobsthal numbers. They are defined by the recurrence $J_n=J_{n-1}+2J_{n-2}$, where $J_0 = 0$, $J_1 = 1$ (\href{http://oeis.org/A001045}{A001045} on OEIS\cite{OEIS}).  The next few are $J_2 = 1$, followed by 3, 5, 11, 21, \ldots. Solving the recurrence gives $J_n = \frac{1}{3}(2^{n} - (-1)^{n})$, and so apart from $J_0$ they are all odd. Using these numbers, we relate a run's terms with its node:

\begin{theorem}\label{runterms} Given a node $(a,b)$ where $a,b>0$ are odd, let $b = a + 2^{k-2}d$ with odd (but not necessarily positive) $d > -a/2^{k-2}$. The corresponding run is then $\{a+2^{k-i}J_{i-1}d\}$ where $i$ goes from 1 to $k$. The run has length $k$, consisting of $k-1 \ge 2$ odd terms followed by a single even term $a+J_{k-1}d$.
\end{theorem}

\begin{proof}
We justify the exponent $k-2$ in $b = a + 2^{k-2}d$ as it counts the number of divisions by 2, which occur for all terms but the two node terms. Thus $k$ denotes run length. We conclude that the first $k-1$ members
\begin{equation}\label{eq:run}
a+2^{k-1}J_0d=a, \quad a+2^{k-2}J_1d=a+2^{k-2}d, \quad \ldots, \quad \ a+2^1J_{k-2}d
\end{equation}
are all odd since $a$ and $d$ are odd by definition, while the last ($k$-th) term $a+2^0J_{k-1}d$ is even.  By the recurrence, each term
after the first two is the average of the two previous ones, a consequence of the definition of a subprime Fibonacci sequence. The condition $d>-a/2^{k-2}$ ensures that
all our terms are positive.
\end{proof}

We refer to Theorem \ref{runterms} for a more run-based system for a cycle. Write out two runs in the style of Eq.\ \ref{eq:run}:
\begin{eqnarray*}
a_1, \ a_1+2^{k_1-2}d_1, \ a_1+2^{k_1-3}d_1, \ \dots ,
 \ a_1+2J_{k_1-2}d_1, \ a_1+J_{k_1-1} d_1 \\
a_2, \ a_2+2^{k_2-2}d_2, \ a_2+2^{k_2-3}d_2, \ \dots ,
 \ a_2+2J_{k_2-2}d_2, \ a_2+J_{k_2-1} d_2
\end{eqnarray*}

Concatenate the two runs. The two terms after the first run will be the first two terms (the node) of the second run. Remembering that $J_n=J_{n-1}+2J_{n-2}$, we can express these terms as
\begin{equation*}
\frac{2a_1+J_{k_1}d_1}{p_2} \mbox{\quad and \quad}
\frac{(p_2+2)a_1+(J_{k_1-1}p_2+J_{k_1})d_1}{p_2q_2}
\end{equation*}
respectively, where $p_2$, $q_2$ are the least prime \textbf{divisors} of the node of the second run. We will reserve the use of `divisors' to the signature terms of nodes.

Curve the two runs into a cycle and denote the divisors of the first run as $p_1$ and $q_1$. As with Eq.\ \ref{eq:system} we fix the length $m$ of the cycle, which is done by fixing the individual run lengths $k_1, k_2$. We now have the four equations
\begin{eqnarray*}
p_2a_2 & = & 2a_1+J_{k_1}d_1 \\
p_2q_2(a_2+2^{k_1-2}d_2) & = & (p_2+2)a_1+(J_{k_1-1}p_2+J_{k_1})d_1 \\
p_1a_1 & = & 2a_2+J_{k_2}d_2 \\
p_1q_1(a_1+2^{k_2-2}d_1) & = & (p_1+2)a_2+(J_{k_2-1}p_1+J_{k_2})d_2
\end{eqnarray*}
giving four linear homogeneous equations as viewed in terms of $a_1$, $d_1$, $a_2$, $d_2$. Subtracting the first and third equations from the second and fourth then removing a factor $p_i$ from each gives
\begin{equation}\label{eq:2rundet}
\begin{vmatrix} 2 & J_{k_1} & -p_2 & 0 \\
1 & J_{k_1-1} & -q_2+1 & -2^{k_2-2}q_2 \\
-p_1 & 0 & 2 & J_{k_2} \\
-q_1+1 & -2^{k_1-2}q_1 & 1 & J_{k_2-1} 
\end{vmatrix}
= 0.
\end{equation}

Expanding and applying the identity $2^{k-1} - J_k = J_{k-1}$ reduces Eq.\ \ref{eq:2rundet} to:
\begin{align*}
2^{k_1+k_2-4}p_1q_1p_2q_2 \quad=\quad & J_{k_1-1}J_{k_2-1}p_1p_2+
J_{k_1}J_{k_2-2}p_1q_2+J_{k_1-2}J_{k_2}q_1p_2+J_{k_1-1}J_{k_2-1}q_1q_2\\
& + J_{k_1}J_{k_2-1}p_1+J_{k_1-1}J_{k_2}q_1+J_{k_1-1}J_{k_2}p_2+ J_{k_1}J_{k_2-1}q_2\\
& +J_{k_1}J_{k_2}-(-1)^{k_1+k_2-4}.
\end{align*}
Note that \emph{any} 2-run cycle's node divisors must satisfy this relation.

This method generalizes to $n$ runs, again where we `curve' the $n$ runs into a cycle so that the $n$-th run gives rise to the 1st run. This gives $2n$ equations in $2n$ unknowns. In general, where $k_i$ is the length of the $i$-th run, $p_i$ and $q_i$ being the divisors of the $i$-th node, and letting $p_0$ refer to $p_n$, the equation for $n$-run cycles may be written as
\begin{equation}\label{eq:nrun}
\prod_{i=1}^n 2^{\sum(k_i-2)}p_{i-1}q_i=
\sum_{\substack{\delta_1, \dotsc, \delta_{2n} \in \{0,1\}, \\ \sum \delta_j \le n, \\ \forall j \le n: \; \delta_j + \delta_{j+n} < 2}}
\prod_{i=1}^nJ_{k_i-\delta_i-\delta_{i+n}}
p_{i-1}^{\delta_i}q_i^{\delta_{i+n}}-(-1)^{\sum(k_i-2)}
\end{equation}
for $n \ge 2$. 

The input for this formula only requires the number of runs in the cycle and the associated \textbf{run configuration}, which is the $n$-tuple of run lengths and thus a concise version of the shape. For example, the run configuration of the 10-cycle of shape $OOEOOEOOOE$ is $(k_1,k_2,k_3) = (3,3,4)$, where $m = \sum k_i = 10$. It is important to note that as with the shapes of cycles, run configurations $(3,4,3)$ and $(4,3,3)$ are identical to $(3,3,4)$ since cycles have no definitive starting nodes; what matters is that the order of run lengths is preserved.

\section{Signature restrictions}
\label{sec:sigres}

Let the terms of an arbitrary cycle be $a_1, b_1, \frac{a_1+b_1}{2}, \frac{a_1+3b_1}{4}, \dotsc$, $a_n, b_n, \frac{a_n+b_n}{2}, \frac{a_n+3b_n}{4}, \dotsc$ where $(a_i,b_i)$ are the cycle's nodes. Let the respective signature be $p_1,q_1,2,2,\dotsc$, $p_n,q_n,2,2,\dotsc$ etc., where $p_i$ corresponds to $a_i$ and $q_i$ corresponds to $b_i$. Each $p_i$, $q_i$ is either 1 or an odd prime; we will refer to these, the divisors of the cycle's nodes, collectively as \textbf{the cycle's divisors}. As already established, $n \ge 2$.

We were able to disprove 1-run cycles and relate the divisors of $n$-run cycles to one another given the run configuration. But even with a given 2-run configuration, we are left with a relation in 4 unknown variables $p_1,q_1,p_2,q_2$ (remember that a cycle is uniquely determined by its signature, from which every $(a_i, b_i)$ can be recovered). Can we further restrict these variables? We already know that at least one of $p_1,q_1,\dotsc,p_n,q_n$ is 1, and that $p_1,q_1,\dotsc,p_n,q_n$ cannot all equal 1 since cycle terms cannot increase indefinitely.

Can we strengthen these results? To motivate another approach, consider the following diagram of term vs. index for the 18-cycle (Figure~\ref{fig:18cycle}), which is composed of four runs:

\begin{figure}[H]
\centering

\begin{tikzpicture}[xscale=0.8,yscale=0.6]

\draw [<->, very thick] (0,14) -- (0,0.1) -- (19.5,0.1);

\fill [black!10] (0.5, 1.3) rectangle (6.5, 6.1);
\fill [black!10] (6.5, 4.5) rectangle (10.5, 8.9);
\fill [black!10] (10.5, 4.1) rectangle (15.5, 9.7);
\fill [black!10] (15.5, 5.3) rectangle (18.5, 4.3);

\tikzstyle{seq} = [shape=circle,fill=black,inner sep=0.5mm]

\path
node (a) at (0, 4.8) [seq] {}
node (b) at (1, 1.3) [seq] {}
node (c) at (2, 6.1) [seq] {}
node (d) at (3, 3.7) [seq] {}
node (e) at (4, 4.9) [seq] {}
node (f) at (5, 4.3) [seq] {}
node (g) at (6, 4.6) [seq] {}
node (h) at (7, 8.9) [seq] {}
node (i) at (8, 4.5) [seq] {}
node (j) at (9, 6.7) [seq] {}
node (k) at (10, 5.6) [seq] {}
node (l) at (11, 4.1) [seq] {}
node (m) at (12, 9.7) [seq] {}
node (n) at (13, 6.9) [seq] {}
node (o) at (14, 8.3) [seq] {}
node (p) at (15, 7.6) [seq] {}
node (q) at (16, 5.3) [seq] {}
node (r) at (17, 4.3) [seq] {}
node (s) at (18, 4.8) [seq] {}
node (t) at (19, 1.3) [seq] {};

\draw[black] (a) -- (b) -- (c) -- (d) -- (e) -- (f) -- (g) -- (h) -- (i) -- (j) -- (k) -- (l) -- (m) -- (n) -- (o) -- (p) -- (q) -- (r) -- (s) -- (t);

\tikzstyle{every node} = [node distance=4mm]

\path
node [left of=a] {48}
node [below of=b] {13}
node [above of=c] {61}
node [below of=d] {37}
node [above of=e] {49}
node [below of=f] {43}
node [below of=g] {46}
node [above of=h] {89}
node [below of=i] {45}
node [above of=j] {67}
node [above of=k] {56}
node [below of=l] {41}
node [above of=m] {97}
node [below of=n] {69}
node [above of=o] {83}
node [above of=p] {76}
node [below of=q] {53}
node [below of=r] {43}
node [above of=s] {48}
node [right of=t] {13};

\draw[snake=brace, segment amplitude=3mm, thick] (0.5,12.5) -- (18.5, 12.5);
\node at (9.5, 13.5) {cycle};
\draw[snake=brace, segment amplitude=3mm, thick] (0.5,11) -- (6.5,11);
\node at (3.5, 12) {$OOOOOE$};
\draw[snake=brace, segment amplitude=3mm, thick] (6.5,11) -- (10.5,11);
\node at (8.5, 12) {$OOOE$};
\draw[snake=brace, segment amplitude=3mm, thick] (10.5,11) -- (15.5, 11);
\node at (13, 12) {$OOOOE$};
\draw[snake=brace, segment amplitude=3mm, thick] (15.5,11) -- (18.5, 11);
\node at (17, 12) {$OOE$};

\draw[dashed] (0.5, 11) -- (0.5, 0.1);
\draw[dashed] (6.5, 11) -- (6.5, 0.1);
\draw[dashed] (10.5, 11) -- (10.5, 0.1);
\draw[dashed] (15.5, 11) -- (15.5, 0.1);
\draw[dashed] (18.5, 11) -- (18.5, 0.1);

\end{tikzpicture}

\caption{18-cycle with runs labeled and local run bounds shaded}

\label{fig:18cycle}

\end{figure}
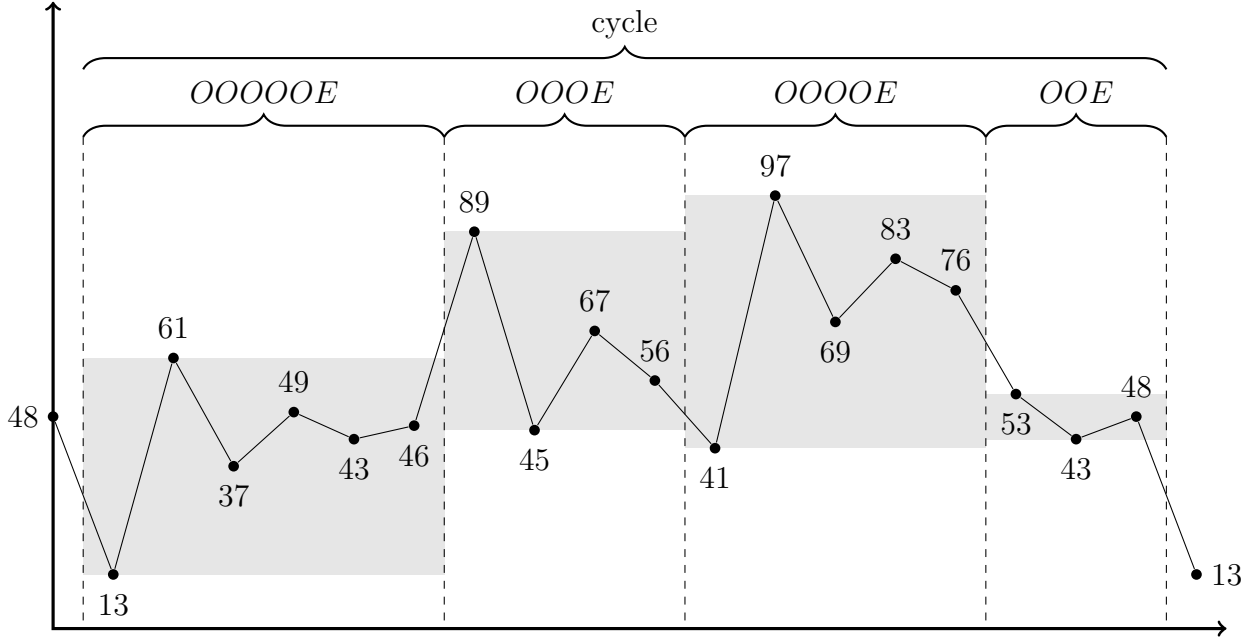

The shaded areas represent bounds on each run and are a consequence of Corollary~\ref{runbounds}. Remembering Corollary~\ref{largestterm} and that if some $p_i$ or $q_i \ne 1$ then the sum of the two terms preceding the corresponding $a_i$ or $b_i$ was divided by at least 3, we provide three stronger results:

\begin{prop}\label{twoarenot1}
At least two of $p_1,q_1,\dotsc,p_n,q_n$ do not equal 1.
\end{prop}

\begin{proof}
Without loss of generality, let all but $p_2$ or $q_2$ be 1.

\textbf{Case 1.} $p_2 \ne 1$. Then $b_1$ and the rest of the terms in the first run are all greater than $a_1$, and since $a_2 > 0$, we have $b_2 > a_1$. Since all other $p_i,q_i = 1$, all terms before $a_1$ are greater than $\frac{1}{2}a_1$. Then $a_1 > a_1$, a contradiction.

\textbf{Case 2.} $q_2 \ne 1$. Then $b_1$ and the rest of the terms in the first run are all greater than $a_1$, and since $a_2 > 2a_1$, we have $b_2 > 0$ and $\frac{a_2 + b_2}{2} > a_1$. Since all other $p_i,q_i = 1$, all terms before $a_1$ are greater than $\frac{1}{2}a_1$. Then $a_1 > a_1$, a contradiction.
\end{proof}

\begin{prop}\label{twoare1}
At least two of $p_1,q_1,\dotsc,p_n,q_n$ equal 1.
\end{prop}

\begin{proof}
Without loss of generality, let exactly one of $p_1,q_1$ be $1$ (the rest are $\ge 3$). Then the corresponding term $a_1$ or $b_1$ is prime and the largest term of the cycle.

\textbf{Case 1.} $p_1=1$. Then $b_1 < \frac{2}{3} a_1$ and the rest of the terms in the run are less than $\frac{5}{6} a_1$. Then $a_2 < \frac{5}{9} a_1$, $b_2 < \frac{25}{54} a_1$, and the rest of the terms of the second run are less than $\frac{55}{108} a_1$. Since all further $p_i$,$q_i$ (if any) are also greater than $3$, they do not increase the maximum. We iterate our bounding to get $b_1 < \frac{163}{324} a_1 < \frac{5}{9} a_1$ and that the rest of terms of the run are less than $\frac{7}{9} a_1$. Then $a_2 < \frac{14}{27} a_1$, $b_2 < \frac{35}{81} a_1$, and the rest of the terms before $a_1$ are less than $\frac{77}{162} a_1 < \frac{1}{2}a_1$. Then $a_1 < a_1$, a contradiction.

\textbf{Case 2.} $q_1=1$. The first run's terms are less than $b_1$ except for $b_1$ itself. Then $a_2 < \frac{2}{3} b_1$, $b_2 < \frac{5}{9} b_1$, and the rest of the terms of the second run are less than $\frac{11}{18} b_1$. Since all further $p_i$,$q_i$ (if any) are also greater than $3$, they do not increase the maximum. Then $a_1 < \frac{11}{27} b_1 < \frac{1}{2} b_1$.  We iterate our bounding to see that $\frac{a_1+b_1}{2} < \frac{3}{4} b_1$ and that the rest of the terms of the run are less than $\frac{7}{8} b_1$. Then $a_2 < \frac{7}{12} b_1$, $b_2 < \frac{35}{72} b_1$, and the rest of the terms before $a_1$ are less than $\frac{77}{144} b_1$. Then $a_1 < \frac{77}{216} b_1$ and $b_1 < \frac{385}{432} b_1 < b_1$, a contradiction.
\end{proof}

\begin{prop}\label{twoare1ifsep}
If there are only two $p_1,q_1,\dotsc,p_n,q_n$ that equal 1, the two cannot be of the form $p_i,q_i$ unless $(p_{i+1},q_{i+1}) = (3,3)$, $(3,5)$, or $(5,3)$.
\end{prop}

\begin{proof}
Without loss of generality, let $p_1,q_1 = 1$ and all further $p_i,q_i \ge 3$. Then $b_1$ is the largest term, with $b_1 < 2a_1$, $\frac{a_1+b_1}{2} < \frac{3}{2} a_1$, and the rest of the run is $< \frac{7}{4} a_1$. Then $a_2 < \frac{7}{4p_2} a_1$, $b_2 < \frac{7(p_2+1)}{4 p_2 q_2} a_1$. Since all further $p_i,q_i$ (if any) are $\ge 3$ then all subsequent terms before $a_1$ are $< \max \left(\frac{7}{4p_2} a_1, \frac{7(p_2+1)}{4 p_2 q_2} a_1 \right)$. Hence $a_1 < \max \left(\frac{7}{2p_2} a_1, \frac{7(r+1)}{2p_2q_2} a_1 \right)$, a contradiction unless $(p_2,q_2) = (3,3)$, $(3,5)$, or $(5,3)$.
\end{proof}

These results are particularly restrictive on cycles of only two nodes, that is, where the only divisors are $p_1,q_1,p_2,q_2$. One might conjecture that

\begin{conj}\label{tworuns}
There are no non-trivial cycles of two runs (i.e., two even terms).
\end{conj}

Because of these results, all that is needed to prove this conjecture is a similar argument against the cases where one of $p_1,q_1$ and one of $p_2,q_2$ are 1, and eliminating the three exceptions of Proposition~\ref{twoare1ifsep}. However, consider the cycle signature $7,1,2,1,5,2,2$. Using these values for $s_1,\dotsc,s_n$ in the earlier system and scaling as in Theorem~\ref{signaturetocycle} gives the cycle candidate $13,51,32,83,23,53,38$, which would work if 51 were prime. Thus to prove that other `cycles' like this similarly fail, the primality test for an unknown set of numbers may be required.

However, it also seems possible that with a requirement of exactly two runs, primes in a signature and terms in a cycle are bounded in some way. After all, longer instances of such a cycle only means that runs take longer to terminate, but since runs are recurrences of averages, the cycle's two nodes' positions relative to each other should be fairly restricted.

\section{Cycles of a given length}
\label{sec:givenlen}

Regardless of whether the preceding argument can be formalized and generalized to cycles of any number of runs, it is still important that the cases involving cycles of shorter lengths are exhausted. How can we do this? Consider what we know:
\begin{itemize}[noitemsep,nolistsep]
\item Relationships between signature terms and between divisors (Eqs.\ \ref{eq:system} and \ref{eq:nrun})
\item Each signature corresponds to a unique potential cycle (Theorem~\ref{signaturetocycle})
\item Non-existence of 1-run cycles (Theorem~\ref{onerun})
\item Restrictions on possible signatures (Propositions~\ref{twoarenot1}~to~\ref{twoare1ifsep})
\end{itemize}
These give a way to determine whether 2-run cycles of a given length exist, which for cycle lengths of 6 to 8 exhaust all possible cycles of that length:

\begin{theorem}\label{6cycle}
There are no 6-cycles.
\end{theorem}

\begin{proof}
A 6-cycle must have shape $OOEOOE$ and therefore a signature $p_1, q_1, 2, p_2, q_2, 2$. Using either Eq.\ \ref{eq:2rundet} with run configuration $(k_1, k_2) = (3,3)$ or the more general Eq.\ \ref{eq:system} for $n=6$, we get
\begin{equation*}
4 p_1 q_1 p_2 q_2 = p_1 p_2 + q_1 q_2 + 3( p_1 q_2 + q_1 p_2 + p_1 + q_1 + p_2 + q_2 ) + 8.
\end{equation*}
By our previous results, exactly two of $p_1,q_1,p_2,q_2$ must be 1. There are only four cases:

\textbf{Case 1.} $p_1,q_1=1$ (equivalent to $p_2,q_2 = 1$). Then $4 p_2 q_2 = 7 (p_2 + q_2 + 2)$. Since we only want solutions over the odd primes, then exactly one of $p_2$ and $q_2$ is 7 and $(p_2,q_2) = (3,7), (7,3)$. These solutions fail by Proposition~\ref{twoare1ifsep}.

\textbf{Case 2.} $p_1,p_2=1$. Then $q_1 q_2 = 2 q_1 + 2 q_2 + 5 \Rightarrow (q_1 - 2)(q_2 - 2) = 9$. The solutions over the odd primes are $(q_1,q_2) = (3,11),(5,5),(11,3)$, though reordering the runs shows $(3,11)$ and $(11,3)$ are equivalent.

\textbf{Case 3.} $q_1,q_2=1$. Then $p_1 p_2 = 2 p_1 + 2 p_2 + 5 \Rightarrow (p_1,p_2) = (3,11),(5,5),(11,3)$, though reordering the runs shows $(3,11)$ and $(11,3)$ are equivalent.

\textbf{Case 4.} $p_1,q_2=1$ (equivalent to $p_2,q_1 = 1$). Then $q_1 p_2 = 4 q_1 + 4 p_2 + 17 \Rightarrow (q_1 - 4)(p_2 - 4) = 33$. The solutions over the odd primes are $(q_1,p_2) = (5,37),(37,5)$.

Substitute the signature values into Eq.\ \ref{eq:system} and solve the system. Since the solution space is 1-dimensional, we can express all the cycle terms $t_1,\dotsc,t_n$ in terms of $t_1$ (even better: let $t_1=1$), and then multiply by the common denominator to get the unique cycle candidate:
\begin{table}[H]
\small
$
\begin{array}{ | c | l | l | l | l | l | l | l |}
\hline
(p_1, q_1, p_2, q_2) & t_2/t_1 & t_3/t_1 & t_4/t_1 & t_5/t_1 & t_6/t_1 & \text{Cycle candidate} & \text{It should be\ldots} \\ \hline
(1,3,1,11) & 3/5 & 4/5 & 7/5 & 1/5 & 4/5 & 5,3,4,7,1,4 & 4,7,11 \text{, not } 4,7,1 \\ \hline
(1,5,1,5) & 1/3 & 2/3 & 1 & 1/3 & 2/3 & 3,1,2,3,1,2 & 2,3,5 \text{, not } 2,3,1 \\ \hline
(3,1,11,1) & 19/9 & 14/9 & 1/3 & 17/9 & 10/9 & 9,19,14,3,17,10 & 19,14,11 \text{, not } 19,14,3 \\ \hline
(5,1,5,1) & 3 & 2 & 1 & 3 & 2 & 1,3,2,1,3,2 & 3,2,5 \text{, not } 3,2,1 \\ \hline
(1,5,37,1) & 11/41 & 26/41 & 1/41 & 27/41 & 14/41 & 41,11,26,1,27,14 & 11,26,37 \text{, not } 11,26,1 \\ \hline
(1,37,5,1) & 1/27 & 14/27 & 1/9 & 17/27 & 10/27 & 27,1,14,3,17,10 & 10,27,37 \text{, not } 10,27,1 \\ \hline
\end{array}
$
\normalsize
\caption{Candidates for a 6-cycle}
\end{table}
The last entry also fails because $1,14$ should be followed by 5, not 3, and because the largest term is not prime. Since all candidates fail, the theorem is proved. \end{proof}

The problem is that the linear system takes divisibility into account, but not divisibility by the \emph{smallest} prime factor, or no division if a sum is already prime. Note that the symmetries above do not always occur; here they arise from both the runs being of shape $OOE$.

There are also two lemmas that can simplify things:

\begin{lemma}\label{smallestterm}
The smallest term in a non-trivial cycle must be a node term (and thus odd), and at least 7.
\end{lemma}

\begin{proof}
Since node terms bound a run's terms, the smallest number of the cycle must also be one of its node terms, which are odd by definition.

Dividing a composite number by its smallest prime factor never produces 1. If 3 is the smallest cycle term, the previous members $a,b$ must add to 3, 6 or 9. Since the same integer cannot be separated by only one term (the sequence $a,b,a$ continues into trivial cycle $a,a,\dotsc$), it follows that the smallest number in the cycle is less than the two preceding members. Hence, $a$ and $b$  are greater than 3, giving the two cases $(a,b) = (5,4), (4,5)$ which are tributary to, but not part of, non-trivial cycles.

If 5 occurs in a sequence, the previous members $a,b$ must add to 5, 10, 15, or 25, and if they are to be greater
than 5, $(a,b) = (6,9), (7,8), (8,7), (9,6), (6,19), \dotsc (19,6)$ (this list would be deduced by the first half of the `direct predecessor' method of Section \ref{sec:nodes}). Calculation shows these are tributary to, but not part of, non-trivial cycles.

Finally, 7 is the smallest member of the 136-cycle, completing the proof.
\end{proof}

Note that this result immediately disqualifies all the 6-cycle candidates. Also, using lower bound arguments omitted here, we can eliminate the three exceptions of Proposition~\ref{twoare1ifsep} for all 2-run cycles:

\begin{lemma}
In a 2-run cycle, $p_1,q_1$ cannot both be 1 and $p_2,q_2$ cannot both be 1.
\end{lemma}

Thus for cycles of longer length, we can apply Theorem~\ref{6cycle}'s method of generating candidates and easily show why they fail. For 7-cycles, the only possible run configuration is $(3,4)$; for 8-cycles, the two possible configurations are $(3,5)$ and $(4,4)$. By considering all configuration cases and proceeding with the method programmatically, Andrew Bremner has shown \cite{Bre13} that:

\begin{theorem}
There are no 2-run cycles of length 30 or less.
\end{theorem}

One can also consider cycles of more than two runs, though they require significantly more casework. Consider $(3,3,4)$, the only valid 3-run configuration for 10-cycles. Bremner has provided the following form of argument:

With Eq.\ \ref{eq:system} for $n=10$, we get
\begin{align*}
16 p_1 q_1 p_2 q_2 p_3 q_3 =&\ 3 p_1 p_2 p_3 + 3 p_1 p_2 q_3 + 9 p_1 q_2 p_3 + 3 p_1 q_2 q_3 + 5 q_1 p_2 p_3 + 9 q_1 p_2 q_3\\
&+ 5 q_1 q_2 p_3 + 3 q_1 q_2 q_3 + 9 p_1 p_2 + 9 p_1 q_2 + 9 p_1 p_3 + 9 p_1 q_3 + 15 q_1 p_2\\
&+ 5 q_1 q_2 + 5 q_1 p_3 + 9 q_1 q_3 + 5 p_2 p_3 + 9 p_2 q_3 + 15 q_2 p_3 + 9 q_2 q_3 + 27 p_1\\
&+ 15 q_1 + 15 p_2 + 15 q_2 + 15 p_3 + 27 q_3 + 44.
\end{align*}

Observe that each summand can be bounded in terms of $p_1 q_1 p_2 q_2 p_3 q_3$ as long as an equivalent condition holds, e.g.,
\begin{equation*}
3 p_1 p_2 p_3 < (43/70) p_1 q_1 p_2 q_2 p_3 q_3 \iff q_1 q_2 q_3 > 210/43.
\end{equation*}
Construct 26 such conditions (one for each summand) so that if they all hold, then
\begin{equation*}
16 p_1 q_1 p_2 q_2 p_3 q_3 < 26(43/70)p_1 q_1 p_2 q_2 p_3 q_3 + 44,
\end{equation*}
which is equivalent to $p_1 q_1 p_2 q_2 p_3 q_3 < 1540$. Either this holds, or one of the 26 conditions is false; for example, $q_1 q_2 q_3 > 210/43$ might not hold. Equivalently, this means at least one of 27 upper-bounding conditions must be satisfied.

By significant casework (aided by eliminating cases with symmetry arguments, requiring the divisors to be odd prime or 1, and restricting the number of 1s), one gets an exhaustive list of candidates for $(p_1, q_1, p_2, q_2, p_3, q_3)$. Most are eliminated as before by noting that the corresponding cycle candidate does not follow sequence rules, e.g., the smallest prime divisor is not divided out. By this process one retrieves the 10-cycle discovered earlier as the only 10-cycle of configuration $(3,3,4)$. Bremner also showed that there are no 9-cycles of configuration $(3,3,3)$ in a similar manner \cite{Bre13}.

Since no cycles of length 8 or less have 3 runs, and because the only 3-run configurations for 9-cycles and 10-cycles are equivalent to $(3,3,3)$ and $(3,3,4)$, respectively, we can definitively state that:

\begin{theorem}
There are no (non-trivial) cycles of length 9 or less. There is only one cycle of length 10, generated by $(127, 509)$.
\end{theorem}

\section{Conclusion}

This paper has explored relatively cursory properties of the subprime Fibonacci sequences: most of our deductions have relied only on elementary number theory, algebra, empirical observations, and diagrams. Of course, this is the way we prefer it; to write the first exposition and let others prove the hard results!

It all returns to the low barrier to playing with these sequences. Surely, similar manipulations will yield new results, but we expect that significantly deeper mathematics will be needed to answer the difficult question of the (non-)existence of divergent sequences and the finitude of cycles, perhaps the kind of mathematics necessary to solve the notorious $3x+1$ problem.

However, there are plenty of questions that seem both computationally and mathematically tractable. Here are some of the more obvious ones:

\begin{itemize}
\item We have shown that there are no 2-run cycles of length 30. How far can this be extended computationally? There are no 3-run cycles of configuration $(3,3,3)$ and only one of $(3,3,4)$. Can you also extend the 3-run procedure and show that the 10-cycle and the 11-cycle are the only 3-run cycles less than a certain length? Procedures to exhaust 4-run or greater cases would also be welcome.
\item Are there \emph{any} other non-trivial cycles? We have found six non-trivial cycles using starting values $a,b$ within the range $1 \le a,b \le 10^{6}$. This could be attacked by increasing the search range or considering more classes (run configurations or otherwise) of cycles.
\item We did not explore if/how divisors and terms are bounded based on the number of runs. Maybe one can prove cycle results in this manner. Similarly, don't immediately accept our abstractions of runs, nodes, and signatures if other approaches are fruitful!
\end{itemize}

We leave the reader with a recent article by Conway on unsettleable arithmetical problems, featuring the $3x+1$ problem and `Collatzian games' \cite{Con13}.  It's a casual warning to not be too occupied with answering the big questions. Regardless, have fun and let us know what you discover.

\bibliographystyle{amsplain}
\bibliography{subprime}

\providecommand{\bysame}{\leavevmode\hbox to3em{\hrulefill}\thinspace}
\providecommand{\MR}{\relax\ifhmode\unskip\space\fi MR }
\providecommand{\MRhref}[2]{%
  \href{http://www.ams.org/mathscinet-getitem?mr=#1}{#2}
}
\providecommand{\href}[2]{#2}
\begin{thebibliography}{10}

\bibitem{OEIS}
\emph{The {O}n-{L}ine {E}ncyclopedia of {I}nteger {S}equences}, published
  electronically at \url{http://oeis.org}. Conway's creeper sequence
  (\href{http://oeis.org/A164338}{A164338}), Jacobsthal numbers
  (\href{http://oeis.org/A001045}{A001045}).

\bibitem{And02}
Paul~J. Andaloro, \emph{The {$3x+1$} problem and directed graphs}, Fibonacci
  Quart. \textbf{40} (2002), no.~1, pp. 43--54.

\bibitem{Bre13}
Andrew Bremner, personal communication, 2013.

\bibitem{Con13}
John~H. Conway, \emph{On unsettleable arithmetical problems}, Amer. Math.
  Monthly \textbf{120} (2013), no.~3, pp. 192--198.

\bibitem{Co99}
Curtis Cooper and Robert~E. Kennedy, \emph{Base 10 {RATS} cycles and
  arbitrarily long base 10 {RATS} cycles}, Applications of {F}ibonacci numbers,
  {V}ol. 8 ({R}ochester, {NY}, 1998), Kluwer Acad. Publ., Dordrecht, 1999,
  pp.~83--93.

\bibitem{Gu83}
Richard~K. Guy, \emph{Unsolved problems: Don't try to solve these problems},
  Amer. Math. Monthly \textbf{90} (1983), no.~1, pp. 35--38+39--41.

\bibitem{Gu89}
Richard~K. Guy, \emph{Conway's {RATS} and other reversals}, Amer. Math. Monthly
  \textbf{96} (1989), no.~5, pp. 425--428.

\bibitem{Kh09}
Tanya Khovanova, \emph{{D}estinies of numbers}, published electronically at
  \url{http://blog.tanyakhovanova.com/?p=155}, July 2009.

\bibitem{Kur07}
Stuart~A. Kurtz and Janos Simon, \emph{The undecidability of the generalized
  {C}ollatz problem}, Theory and applications of models of computation, Lecture
  Notes in Comput. Sci., vol. 4484, Springer, Berlin, 2007, pp.~542--553.

\bibitem{La10}
Jeffrey~C. Lagarias (ed.), \emph{The ultimate challenge: the {$3x+1$} problem},
  American Mathematical Society, Providence, RI, 2010.

\bibitem{Si05}
John Simons and Benne de~Weger, \emph{Theoretical and computational bounds for
  {$m$}-cycles of the {$3n+1$}-problem}, Acta Arith. \textbf{117} (2005),
  no.~1, pp. 51--70.

\bibitem{Urv00}
Tanguy Urvoy, \emph{Regularity of congruential graphs}, Mathematical
  foundations of computer science 2000 ({B}ratislava), Lecture Notes in Comput.
  Sci., vol. 1893, Springer, Berlin, 2000, pp.~680--689.

\end{thebibliography}

\end{document}